\documentclass[12pt, final]{amsart}

\usepackage{geometry}
\usepackage{multirow}
\usepackage{graphicx}
\usepackage{float}
\usepackage{amssymb}
\usepackage{verbatim} 
\usepackage{amsmath}
\usepackage{amsthm}
\usepackage{color}
\usepackage{epstopdf}
\usepackage{epsfig}
\usepackage[all]{xy}
\usepackage{enumerate}
\usepackage{hyperref}
\usepackage{multicol}

\newcommand{\vol}{{\rm vol}}
\newtheorem*{defi}{Definition}
\newtheorem*{thm}{Theorem}
\newtheorem*{prop}{Proposition}

\begin{document}
\title{Area minimizing unit vector fields on antipodally punctured unit 2-sphere}

\author{Fabiano G. B. Brito$^1$}
\author{Jackeline Conrado$^2$}
\author{Icaro Gon\c calves$^3$}
\author{Adriana V. Nicoli$^4$}

\thanks{The second and fourth authors was financed by Coordena\c c\~ao de Aperfei\c coamento de Pessoal de N\'ivel Superior- Brasil (CAPES) - Finance Code 001.}

\address{Centro de Matem\'{a}tica, Computa\c c\~{a}o e Cogni\c c\~{a}o, 
Universidade Federal do ABC, Santo Andr\'{e}, 09210-170, Brazil.}
\email{fabiano.brito@ufabc.edu.br, icaro.goncalves@ufabc.edu.br}

\address{Dpto. de Matem\'{a}tica, Instituto de Matem\'{a}tica e Estat\'{i}stica, 
Universidade de S\~{a}o Paulo, R. do Mat\~{a}o 1010, S\~{a}o Paulo-SP,
05508-900, Brazil.}
\email{avnicoli@ime.usp.br, jconrado@usp.br}

\subjclass[2019]{}

\begin{abstract} 
We provide a lower value for the volume of a unit vector field tangent to an antipodally Euclidean sphere $\mathbb{S}^2$ depending on the length of an ellipse determined by the indexes of its singularities. 
\end{abstract}

\maketitle

\section*{\it In memory of Amine Fawaz}
 
\section{Introduction and main results}

Inspired by \cite{BorGilM} and \cite{Faw}, we establish sharp lower bounds for the total area of unit vector fields on antipodally punctured Euclidean sphere $\mathbb{S}^2$, and these values depend on the indexes of their singularities.

\begin{thm}\label{Main1} Let $\vec{v}$ be a unit vector field defined on $M = \mathbb{S}^{2} \backslash \left\{N,S\right\}$. If $k = \max\left\{ I_{\vec{v}}(N), I_{\vec{v}}(S)\right\}$, then
	\[\vol(\vec{v}) \geq \pi L(\varepsilon_k), \]
	where $L(\varepsilon_k)$ is the length of the ellipse $\frac{x^2}{k^2}+\frac{y^2}{(k - 2)^2} = 1$ with $k > 2$ and $I_{\vec{v}}(P)$ stands for the Poincar\'e index of $\vec{v}$ around $P$.
\end{thm}

This is a natural extension of the theorem proved in \cite{BCJ} by P. Chacon, D. Johnson and the first author. In \cite{BCJ}, a general lower-bound for area of unit vector fields in $\mathbb{S}^2 \backslash \{N,S\}$ is established. It turns out to be the area of north-south vector field with both indexes equal to $1$. In this context, our theorem provides certain lower bounds for each class of index.

We also exhibit minimizing vector fields $\vec{v}_k$ within each index class. These fields have areas given essentially by the length of ellipses depending just on the indexes in $N$ and $S$.

\section{Preliminaries and the Proof of Theorem}

Let $M = \mathbb{S}^{2} \backslash \left\{N,S\right\}$ be the Euclidean sphere in which two antipodal points $N$ and $S$ are removed. Denote by $g$ the usual metric of $\mathbb{S}^2$ induced from $\mathbb{R}^3$, and by $\nabla$ the Levi-Civita connection associated to $g$. Consider the oriented orthonormal local frame $\left\{ e_1, e_2\right\}$ on $M$, where $e_2$ is tangent to the meridians and $e_1$ to the parallels. Let $\vec{v}$ be a unit vector field tangent to $M$ and consider another oriented orthonormal local frame $\left\{\vec{v}^{\perp},\vec{v}\right\}$ on $M$ and its dual basis  $\left\{ \omega_1, \omega_2 \right\}$ compatible with the orientation of $\left\{ e_1, e_2\right\}$.

In dimension $2$, the volume of $\vec{v}$ is given by
\begin{eqnarray}\label{volreduces}
\vol(\vec{v}) = \int_{\mathbb{S}^2}{\sqrt{1+ \gamma^2 + \delta^2}}\nu,
\end{eqnarray}
where $\gamma=g(\nabla_{\vec{v}}\vec{v}, \vec{v}^{\perp})$ and $\delta = g(\nabla_{v^{\perp}}\vec{v}^{\perp}, \vec{v})$ are the geodesic curvatures associated to $\vec{v}$ and $\vec{v}^{\perp}$, respectively.

Let $\mathbb{S}^1_{\alpha}$ be the parallel of $\mathbb{S}^2$ at latitude $\alpha \in (-\frac{\pi}{2}, \frac{\pi}{2})$ and $\mathbb{S}^1_{\beta}$ be the meridian of $\mathbb{S}^2$ at longitude $\beta \in (0, 2\pi)$. 

\begin{prop}\label{integrando_volume} 
Let $\theta \in [0,\pi/2]$ be the oriented angle from $e_2$ to $\vec{v}$. If $\vec{v}=(\cos \theta) e_1 + (\sin \theta) e_2$ and $\vec{v}^{\perp}=(-\sin \theta)e_1 + (\cos \theta) e_2$, then
\[1 + \gamma^2 + \delta^2 = 1 + (\tan \alpha + \theta_1)^2 + \theta_2^2,\]
where $\theta_1 = d\theta(e_1)$, $\theta_2 = d\theta(e_2)$.
\end{prop}
\begin{proof} 
We have
\begin{equation}\begin{array}{rcl} 
\label{def_gamma}
\gamma & = & g\left( \nabla_{\vec{v}}\vec{v}, \vec{v}^{\perp}\right) \\
	   & = & g\left( \nabla_{(\cos \theta) e_1 + (\sin \theta) e_2} \left[(\cos \theta) e_1 + (\sin \theta) e_2 \right], \vec{v}^{\perp}  \right) \\				
	   & = & g\left( \nabla_{(\cos \theta) e_1} (\cos \theta) e_1, \vec{v}^{\perp}   \right) 
				+ g\left( \nabla_{(\sin \theta) e_2} (\cos \theta) e_1, \vec{v}^{\perp} \right) \\ 
	   & + & g\left( \nabla_{(\cos \theta) e_1} (\sin \theta)e_2, \vec{v}^{\perp} \right)
				+ g\left( \nabla_{(\sin \theta) e_2} (\sin \theta) e_2, \vec{v}^{\perp} \right)

\end{array}
\end{equation}
and
\begin{equation}\begin{array}{rcl} 
\label{def_delta}
\delta & = & g\left( \nabla_{\vec{v}^{\perp}}\vec{v}^{\perp}, \vec{v}\right) \\
	   & = & g\left( \nabla_{(-\sin \theta) e_1 + (\cos \theta) e_2} \left[-(\sin \theta) e_1 + (\cos \theta) e_2 \right], \vec{v}  \right) \\ 
	   & = & g\left( \nabla_{(-\sin \theta) e_1} (-\sin \theta) e_1, \vec{v}   \right) 
				+ g\left( \nabla_{(\cos \theta) e_2} (-\sin \theta) e_1, \vec{v} \right) \\
	   & + & g\left( \nabla_{(-\sin \theta) e_1} (\cos \theta)e_2, \vec{v} \right)
				+ g\left( \nabla_{(\cos \theta) e_2} (\cos \theta) e_2, \vec{v} \right).
\end{array}				
\end{equation}
We write $\gamma$ and $\delta$ as the following sums
\[\gamma = A + B + C + D \enspace\text{and}\enspace \delta = A^{\prime} + B^{\prime} + C^{\prime} + D^{\prime},\]
with
\[ 
\begin{array}{rclrcl}
 A & = & g\left( \nabla_{(\cos \theta) e_1} (\cos \theta) e_1, \vec{v}^{\perp}  \right),
  & B & = & g\left( \nabla_{(\sin \theta) e_2} (\cos \theta) e_1, \vec{v}^{\perp} \right), \\
 C & = & g\left( \nabla_{(\cos \theta) e_1} (\sin \theta)e_2, \vec{v}^{\perp} \right), 
  & D & = & g\left( \nabla_{(\sin \theta) e_2} (\sin \theta) e_2, \vec{v}^{\perp} \right) 
\end{array}
\]
and
\[ 
\begin{array}{rclrcl}
A^{\prime} & = & g\left( \nabla_{(-\sin \theta) e_1} (-\sin \theta) e_1, \vec{v}   \right), & B^{\prime} 
& = & g\left( \nabla_{(\cos \theta) e_2} (-\sin \theta) e_1, \vec{v} \right) \\ 
 C^{\prime} & = & g\left( \nabla_{(-\sin \theta) e_1} (\cos \theta)e_2, \vec{v} \right), 
& D^{\prime} & = & g\left( \nabla_{(\cos \theta) e_2} (\cos \theta) e_2, \vec{v} \right). 
\end{array}
 \]

First observe that $\tan \alpha = g\big( \nabla_{e_1}e_1, e_2 \big)$ and $\nabla_{e_2}e_2 = 0$. By an elementary computation 
we obtain 
\[
\begin{array}{rclrcl}
A & = & \sin^2\theta (\cos \theta) \theta_1 + \cos^3\theta \tan \alpha, & B & = & (\sin^3 \theta) \theta_2, \\		
C & = & (\cos^3\theta)\theta_1 + \sin^2 \theta \cos \theta \tan \alpha, & D & = & (\sin \theta \cos^2\theta) \theta_2 
\end{array}
\]
and
\[
\begin{array}{rclrcl}
A^{\prime} & = & (\sin \theta \cos^2 \theta) \theta_1 + \sin^3\theta \tan \alpha, & B^{\prime} & = & (-\cos^3 \theta) \theta_2, \\
C^{\prime} & = & (\sin^3\theta)\theta_1 + \sin \theta \cos^2 \theta \tan \alpha,  & D^{\prime} & = & -(\sin^2 \theta \cos \theta) \theta_2.
\end{array}
\] 
Moreover,
\[\begin{array}{rcl}
\gamma & = & (\cos^3\theta \tan \alpha + \sin^2 \theta \cos \theta \tan \alpha) + (\sin^2\theta (\cos \theta) \theta_1 + (\cos^3\theta)\theta_1 ) \\
	   & + & ((\sin^3 \theta) \theta_2 + (\sin \theta \cos^2\theta) \theta_2) \\
       & = & \cos \theta \tan \alpha + (\cos \theta) \theta_1 + (\sin \theta) \theta_2 \\
       & = & \cos \theta (\tan \alpha + \theta_1) + (\sin \theta) \theta_2 
\end{array}
\]
and
\[
\begin{array}{rcl}
	\delta & = &  (\sin^3\theta \tan \alpha +  \sin \theta \cos^2 \theta \tan \alpha) + ((\sin \theta \cos^2 \theta) \theta_1 + (\sin^3\theta)\theta_1) \\
		   & + & ((-\cos^3 \theta) \theta_2-(\sin^2 \theta \cos \theta) \theta_2) \\
		   & = & \sin \theta \tan \alpha + (\sin \theta) \theta_1 - (\cos \theta) \theta_2 = \sin \theta (\tan \alpha + \theta_1) - (\cos \theta) \theta_2. 
\end{array}
\]
Finally,
\begin{equation}\label{curva_geode_tangente}
\gamma = \cos \theta (\tan \alpha + \theta_1) + (\sin \theta) \theta_2, 
\end{equation}
\begin{equation}\label{curva_geode_ortogonal}
\delta = \sin \theta (\tan \alpha + \theta_1) - (\cos \theta) \theta_2.
\end{equation}
From equations (\ref{curva_geode_tangente}) and (\ref{curva_geode_ortogonal}), we find
\[
\begin{array}{rcl}
 1 + \gamma^2 + \delta^2 & = & 1 + \left(\cos \theta (\tan \alpha + \theta_1) + (\sin \theta) \theta_2 \right)^2 +  \left( \sin \theta (\tan \alpha + \theta_1) - (\cos \theta) \theta_2 \right)^2 \\
						 & = & 1 + \cos^2 \theta (\tan \alpha + \theta_1)^2 + (\sin \theta)^2 \theta_2^2 + \sin^2 \theta (\tan \alpha + \theta_1)^2 + (\cos^2 \theta) \theta_2^2 \\
						 & = & 1 + (\tan \alpha + \theta_1)^2 + \theta_2^2.
\end{array}
\]
Therefore,
\[
1 + \gamma^2 + \delta^2 = 1 + (\tan \alpha + \theta_1)^2 + \theta_2^2.
\]
\end{proof}


Our proposition allows us to rewrite the volume functional as an integral depending on the latitude $\alpha$ and the derivatives of $\theta$
\begin{equation}\label{vol_usando_novo_integrando}
\vol(\vec{v})= \int_{M}{\sqrt{1+ \left( \tan \alpha + \theta_1\right)^2 + \theta_2^2}}.
\end{equation}

\begin{proof}[Proof of Theorem]
Given $\varphi$ such that $0 \leq \varphi \leq 2\pi$,
{\fontsize{10}{10}\selectfont 
\[1+ \left( \tan \alpha + \theta_1\right)^2 + \theta_2^2 =
\left( \cos\varphi + \sin\varphi \left[ \sqrt{\left(\tan \alpha + \theta_1\right)^2 + \theta_2^2}\right] \right)^2
+ 
\left( -\sin\varphi + \cos \varphi \left[ \sqrt{\left(\tan \alpha + \theta_1\right)^2 + \theta_2^2}\right] \right)^2.
\]}
Hence,
{\fontsize{11}{11}\selectfont 
	\[
\vol(\vec{v}) =  \int_{M}{\sqrt{\left( \cos\varphi + \sin\varphi \left[ \sqrt{\left(\tan \alpha + \theta_1\right)^2 + \theta_2^2}\right] \right)^2
+ 
\left( -\sin\varphi + \cos \varphi \left[ \sqrt{\left(\tan \alpha + \theta_1\right)^2 + \theta_2^2}\right] \right)^2}}.
\]}
Remember that
\[ 
1+ \left( \tan \alpha + \theta_1\right)^2 + \theta_2^2 \geq 
1+ \left( \tan \alpha + \theta_1\right)^2,
\]
implies
\[ 
\sqrt{1+ \left( \tan \alpha + \theta_1\right)^2 + \theta_2^2} \geq 
\sqrt{1+ \left( \tan \alpha + \theta_1\right)^2}.
\]
From the general inequality, $\sqrt{a^2 + b^2} \geq |a\cos \varphi + b \sin \varphi|$, for any $a$, $b$, $\varphi \in \mathbb{R}$, we have
\[
\sqrt{1+ \left( \tan \alpha + \theta_1\right)^2} \geq \left|\cos \varphi + \sin \varphi(\tan(\alpha)+\theta_1) \right|. 
\]
Therefore,
\begin{equation}\label{vol_theta2_zero}
\begin{array}{rcl}
\vol(\vec{v}) & \geq & \int_{M}{\sqrt{\big( \cos\varphi + \sin\varphi |\tan \alpha + \theta_1|\big)^2 + \big(-\sin\varphi + \cos \varphi |\tan \alpha + \theta_1| \big)^2}} \\
			  & \geq & \int_{M}{\cos \varphi + \sin \varphi \left| \tan \alpha + \theta_1 \right|}.
\end{array}
\end{equation}
This inequality is valid for all $\varphi$ such that $0 \leq \varphi \leq 2\pi$. 

As a next step one consider the following conditions:
\begin{enumerate}[i)]
\item 
\[\varphi_k(\alpha) = \arctan \left( \tan\alpha + \frac{k-1}{\cos \alpha}\right);\]
\item
\[\tan\big(\varphi_k(\alpha)\big) = \tan\alpha + \frac{k-1}{\cos \alpha}.\]
\end{enumerate}
Replacing these conditions in equation (\ref{vol_theta2_zero}) we find
\begin{equation}\label{volume_com_alfa}
\vol(\vec{v}) \geq \int_{M}{ \left(\cos \big( \varphi_k(\alpha)\big) + \sin \big(\varphi_k(\alpha) \big)\right) \left| \tan \alpha + \theta_1 \right| \nu}.
\end{equation}
Condition (i) provides that 
\[
\cos \big(\varphi_k(\alpha)\big) = 
\frac{\cos\alpha}{\sqrt{1 + (k-1)^2 + 2(k-1)\sin\alpha}}, \hspace{0.2cm} - \frac{\pi}{2} \leq \alpha \leq \frac{\pi}{2}, 
\]
\[
\sin \big(\varphi_k(\alpha)\big) = \frac{ k-1 + \sin\alpha}{\sqrt{1 + (k-1)^2 + 2(k-1)\sin\alpha}}, \hspace{0.2cm} - \frac{\pi}{2} \leq \alpha \leq \frac{\pi}{2}. 
\]
Thus, as the second part of the inequation (\ref{volume_com_alfa}) is equal to
{\fontsize{10}{10}\selectfont 
\begin{equation}\label{***}
\lim_{\alpha_0 \to - \frac{\pi}{2}} \left[\int_{\alpha_0}^{\frac{\pi}{2}}\int_0^{2\pi}\left( \frac{\cos\alpha}{\sqrt{1 + (k-1)^2 + 2(k-1)\sin\alpha}} + \frac{ k-1 + \sin\alpha}{\sqrt{1 + (k-1)^2 + 2(k-1)\sin\alpha}} \left| \tan \alpha + \theta_1 \right|\right) \right]d\beta d\alpha.
\end{equation}}

Remember that Cartan's connection form $\omega_ {12}$ is given by 
\[
\omega_{12} = \delta \omega_1 + \gamma \omega_2,
\]
where $\{ \omega_1, \omega_2 \}$ is dual basis of $\{\vec{v}^\perp, \vec{v} \}$. Let $i^* : \mathbb{S}^1_{\alpha} \hookrightarrow \mathbb{S}^2$ be the inclusion map, and $e_1= \sin \theta \vec{v}^{\perp} + \cos\theta \vec{v}$, then
\[i^*(\omega_{12})(e_1) = \delta \sin\theta + \gamma\cos \theta.\]
From equations (\ref{curva_geode_tangente}) and (\ref{curva_geode_ortogonal}),
we have
\begin{eqnarray}\nonumber
i^*(\omega_{12})(e_1) &=& \sin\theta \left[\sin \theta \big(\tan\alpha+\theta_1\big)-\cos\theta(\theta_2)\right] + \cos \theta \left[\cos\theta \big(\tan\alpha+\theta_1\big) +\sin\theta (\theta_2)\right]\\ \nonumber
& = &  \tan\alpha + \theta_1.
\end{eqnarray}

Thus, from (\ref{***}) 
\begin{equation}\label{integralcomlimitemenor}
\vol(\vec{v})
\geq \lim_{\alpha_0 \to -\frac{\pi}{2}} \left( \int_{\alpha_0}^{\frac{\pi}{2}} \left( \int_{0}^{2\pi}\left(\frac{\cos \alpha + \big((k-1) + \sin \alpha\big)i^*(\omega_{12})(e_1)}{\sqrt{1+(k-1)^2 + 2(k-1)\sin \alpha}}\right) d\beta\right) d\alpha\right).
\end{equation}
In order to compute the integral of $i^*\omega_{12}$ over the parallel of $\mathbb{S}^2$ at constant latitude $\alpha$, we follow the same arguments in the proof Theorem 1.1 of \cite{BCJ}. 
Denote by $\omega$ the connection form $\omega_{12}$ and 
\[
\mathbb{S}^2_{\alpha} = \{(x,y,z \in \mathbb{R}^3 ; z \geq \sin \alpha, \alpha_0 \leq \alpha \leq \frac{\pi}{2})\}.
\]
The  $2$-form $d\omega$ is given by
\[
d\omega = \omega_1 \wedge \omega_2.
\]
A simple application of Stokes' theorem implies that
\[
\int_{\mathbb{S}^2_{\alpha}}d\omega = 2\pi \big(I_{N}(\vec{v}) \big)- \int_{\mathbb{S}^1_{\alpha}} i^*\omega_{12}.
\]
Suppose that $I_{N}(\vec{v}) = \sup \{I_{N}(\vec{v}), I_{S}(\vec{v})\} = k$, we obtain
\begin{equation}\label{integral_pullback}
\int_{\mathbb{S}^1_{\alpha}}i^*(\omega_{12})(e_1)d\beta = 2\pi k - \mbox{Area}\big( \mathbb{S}^2_{\alpha}\big) = 2\pi k - 2\pi\big( 1 - \sin \alpha \big) = 2\pi\big( k - 1 + \sin \alpha \big).
\end{equation}
From inequation (\ref{integralcomlimitemenor}), 
\[
\begin{array}{rcl}
\vol(\vec{v}) & \geq & \lim_{\alpha_0 \to -\frac{\pi}{2}} \left( \int_{\alpha_0}^{\frac{\pi}{2}} \left( \int_{0}^{2\pi}\left(\frac{\cos \alpha + \big((k-1) + \sin \alpha\big)i^*(\omega_{12})(e_1)}{\sqrt{1+(k-1)^2 + 2(k-1)\sin \alpha}}\right) d\beta\right) d\alpha\right) \\
			  & = & \lim_{\alpha_0 \to -\frac{\pi}{2}}  \int_{\alpha_0}^{\frac{\pi}{2}} \left( \frac{\cos \alpha}{\sqrt{1+(k-1)^2 + 2(k-1)\sin \alpha}}\int_{0}^{2\pi}d\beta + \frac{\big((k-1) + \sin \alpha\big)}{\sqrt{1+(k-1)^2 + 2(k-1)\sin \alpha}}  \int_{\mathbb{S}^1_{\alpha}}i^*\omega_{12} \right) d\alpha \\
			  & = & \lim_{\alpha_0 \to -\frac{\pi}{2}}  \int_{\alpha_0}^{\frac{\pi}{2}} \left( \frac{2\pi\cos^2 \alpha}{\sqrt{1+(k-1)^2 + 2(k-1)\sin \alpha}} 
+ \frac{2\pi\big((k-1) + \sin \alpha\big)^2}{\sqrt{1+(k-1)^2 + 2(k-1)\sin \alpha}}\right) d\alpha,
\end{array}
\]
where the last inequality is obtained from (\ref{integral_pullback}). Therefore,
\[
\vol(\vec{v})
\geq 2 \pi \lim_{\alpha_0 \to -\frac{\pi}{2}}  \int_{\alpha_0}^{\frac{\pi}{2}} \left(\frac{\cos^2 \alpha + \big((k-1) + \sin \alpha\big)^2}{\sqrt{1+(k-1)^2 + 2(k-1)\sin \alpha}}\right)d\alpha.
\]
Analogously,
\[
\vol(\vec{v})
\geq 2 \pi \lim_{\alpha_0 \to -\frac{\pi}{2}}  \int_{\alpha_0}^{\frac{\pi}{2}}\left({\sqrt{1+(k-1)^2 + 2(k-1)\sin \alpha}}\right)d\alpha.
\]
A trigonometrical identity give us
\[
\vol(\vec{v})
\geq 2 \pi \int_{-\frac{\pi}{2}}^{\frac{\pi}{2}}{\sqrt{(k-2)^2 + 4(k-1)\sin^2 \left(\frac{\alpha}{2} + \frac{\pi}{4} \right)}}d\alpha.
\]

Assume that $t = \frac{\alpha}{2} + \frac{\pi}{4}$, then 
\begin{equation}\label{volumenavariavelt}
\vol(\vec{v})
\geq 4 \pi \int_{0}^{\frac{\pi}{2}}{\sqrt{(k-2)^2 + 4(k-1)\sin^2 t}}dt.
\end{equation}
Consider $k > 2$ and an ellipse $\varepsilon_{k}$ given by 
\[\frac{x^2}{k^2} + \frac{y^2}{(k-2)^2} = 1.\]
Let $\mu$ be a parametrization for $\varepsilon_{k}$ defined by
$\mu(t) = (k \cos t, (k-2)\sin t)$.
Its length is
\begin{equation}\label{compr_elipse}
L(\varepsilon_k) = 4\int_{0}^{2\pi}\left(\sqrt{ (k-2)^2 + 4(k-1)\sin^2 t} \right)dt.
\end{equation}
Therefore,
\[
\vol(\vec{v})
\geq \pi L(\varepsilon_k).
\]
\end{proof}

\section{Area-minimizing vector fields $\vec{v}_k$ on $M$}

Now we are going to exhibit a family of unit vector fields attaining the lower value from the main Theorem. At the end of this Section, we also show a geometric interpretation of the areas of these vector fields.

Our previous computations imply 
\[
\begin{array}{rcl}
\vol(\vec{v}) & = & \int[1+\theta_2^2+(\tan\alpha+\theta_1)^2]^{\frac{1}{2}} \\
& \geq & \int[1+(\tan\alpha+\theta_1)^2]^{\frac{1}{2}} \\
& \geq & \int|\cos \varphi+\sin\varphi(\tan\alpha+\theta_1)|. \\
\end{array}
\] 
Assuming $\vol(\vec{v})=\pi L(\varepsilon_k)$, we have  $\theta_2 =0$ and $\cos\varphi(\tan\alpha+\theta_1)=\sin\varphi$, where $\varphi \in \mathbb{R}$. Then $\theta_1=\tan\varphi-\tan\alpha$. Finally $\varphi=\varphi(\alpha)=\arctan\left(\tan\alpha+\frac{k-1}{\cos\alpha}\right)$, which implies $\theta_1=\frac{k-1}{\cos\alpha}$. 

We may summarize this discution in a definition including $p\mapsto e_2(p)$, which is a minimal vector field tangent to the meridians of $\mathbb{S}^2\backslash\{N,S\}$, see \cite{BorGilM} and \cite{BCJ}.

\begin{defi}
 Let $k$ be a positive integer, $k\neq2$ and define:
\begin{enumerate}
	\item  $\vec{v}_1(p) = \vec{e}_2(p)$, if $k=1$;
	\item  $\vec{v}_k(p) = \cos \theta(p)\vec{e}_1(p) + \sin \theta(p)\vec{e}_2(p)$, if $k>2$, where $
	\theta: \mathbb{S}^2 \backslash \left\{N, S \right\} \to \mathbb{R}$ satisfy 
	\[\theta_1 (p) = \frac{k-1}{\sqrt{x^2 + y^2}}.\]
\end{enumerate}
\end{defi}

Notice that $\theta$ has constant variation along the parallel $x^2 + y^2 = \cos\alpha$, with $\alpha\in \left(-\frac{\pi}{2},\frac{\pi}{2}\right)$ constant, and this includes the case where $k=1$.

If we use spherical coordinates $(\beta, \alpha)$ so that $p=(\cos\alpha\cos\beta, \cos\alpha\sin\beta, \sin\alpha)$, we can say that the vector $\vec{v}_k$ spins at a constant speed of rotation along the parallel $\alpha$. Moreover, $\vec{v}_k$ gives exactly $k-1$ turns when it passes the $\alpha$ parallel, with respect to the referential $\left\{\vec{e}_1, \vec{e} _2 \right\} $, and it gives $k$ turns with respect to a fixed polar referential, in this case, $\theta_1 (p) = \frac{k-1}{\cos \alpha}$.

In the case of a vector field  with two singularities at $N$ and $S$ with indexes $4$ and $-2$, respectively, Figures \ref{indicenorte} and \ref{indicesul} provide a visual representation for the case $k=4$
\[
\vol(\vec{v})
\geq4\pi\int_0^{2\pi}\sqrt{4+12sin^2t}\,dt
\]
\begin{multicols}{2}	
	\begin{figure}[H]
		\centering
		\includegraphics[height=6cm]{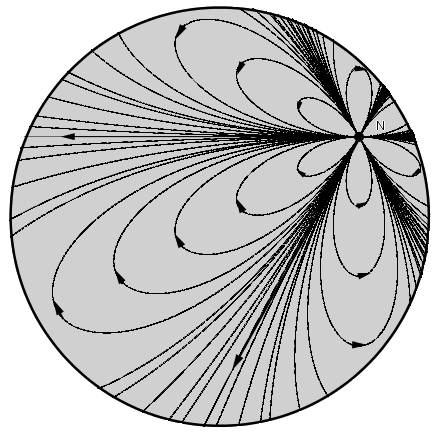}
		\caption{Singularity of index $4$ in north pole}
		\label{indicenorte}
	\end{figure}
	\begin{figure}[H]
		\centering
		\includegraphics[height=6cm]{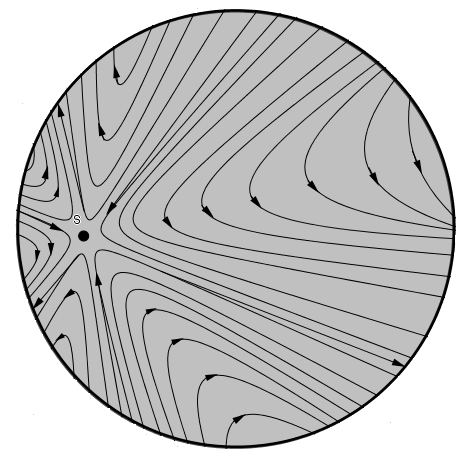}
		\caption{Singularity of index $-2$ in south pole}
		\label{indicesul}
	\end{figure}
\end{multicols}

Finally, for the minimal vector field $\vec{v}_k$, its image  $\vec{v}_k(\mathbb{S}^2\backslash\{N,S\})$ is a surface in $T^1\mathbb{S}^2\backslash\{N,S\}$. Using the well-known Pappus-Guldin's Theorem, its area  equals the area of a suitable ellipsoid of revolution, see the Figure \ref{pneuzinho}.
\begin{figure}[H]
	\centering
	\includegraphics[height=7cm]{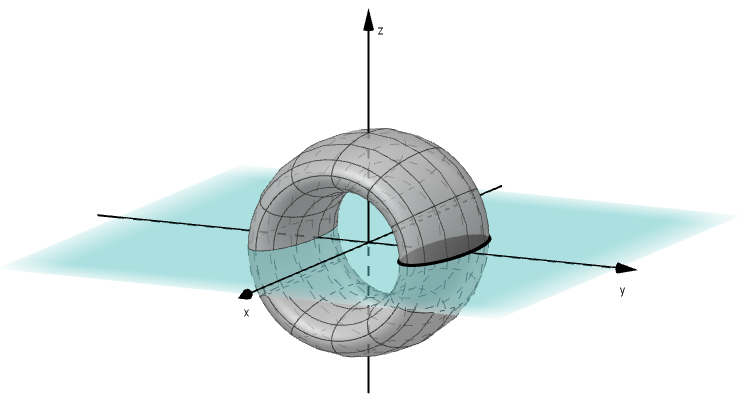}
	\caption{Elippsoid of revolution}
	\label{pneuzinho}
\end{figure}

\section*{Acknowledgment}

We would like to thank professor Giovanni Nunes for carefully reading and commenting a previous version of this manuscript.

\end{document}